\documentclass[11pt]{article}

\usepackage{pstricks}
\usepackage{float}
\usepackage{amsmath,amssymb}
\usepackage{url}
\usepackage{enumerate}
\usepackage{graphicx}
\usepackage{amsthm}

\usepackage{version}

\newtheorem{lemma}{Lemma}[section]

\newtheorem{theorem}{Theorem}[section]
\newtheorem{corollary}{Corollary}[section]
\newtheorem{remark}{Remark}[section]
\newtheorem{example}{Example}[section]

\newcommand{\bX}{\mathbf{X}}
\newcommand{\bY}{\mathbf{Y}}
\newcommand{\by}{\mathbf{y}}

\begin{document}

\noindent {\bf\Large A note on the normal approximation error for randomly weighted self-normalized sums}
\\[3ex]\\[5ex]
{\large Siegfried H\"ormann\footnote{Research supported by the Banque Nationale de Belgique and the Communaut\'e fran\c{c}aise de Belgique - Actions de Recherche Concert\'ees.} and Yvik Swan\footnote{Research supported by a Mandat de Charg\'e de
Recherche from the Fonds National de la Recherche Scientifique, Communaut\'e fran\c{c}aise de Belgique.}}
\\[1ex]
{\em D\'epartement de Math\'ematique, Universit\'e Libre de Bruxelles, Bd.\ Triomphe, CP210, 1050 Brussels, Belgium.\footnote{E-mail: \url{shormann@ulb.ac.be} and 
\url{yvswan@ulb.ac.be}}}\\
\bigskip

\begin{abstract}
Let $\bX=\{X_n\}_{n\geq 1}$ and  $\bY=\{Y_n\}_{n\geq 1}$ be two independent random sequences. We obtain rates of convergence to the normal law of randomly weighted self-normalized sums
$$
\psi_n(\bX,\bY)=\sum_{i=1}^nX_iY_i/V_n,\quad V_n=\sqrt{Y_1^2+\cdots+Y_n^2}.
$$ 
These rates are seen to hold for the convergence of a number of important statistics, such as for instance Student's $t$-statistic or  the empirical correlation coefficient.
\end{abstract}

\section{Introduction}
Let $\bX=\{X_n\}_{n\geq 1}$ and  $\bY=\{Y_n\}_{n\geq 1}$ be two random sequences. In this paper we
investigate the rate of convergence to the normal distribution of the randomly weighted self-normalized sums
\begin{equation}\label{psi}
\psi_n=\psi_n(\bX,\bY)=\sum_{i=1}^nX_iY_i/V_n,\quad V_n=\sqrt{Y_1^2+\cdots+Y_n^2}.
\end{equation}


The random variables $\psi_n$   appear in some important statistics.  For example, when testing the null that the mean of a population $Y$ is equal to 0, one uses the  Student $t$-statistic $$T_n=\frac{\sqrt{n}\bar{Y}}{\left(\frac{1}{n-1}\sum_{i=1}^n(Y_i-\bar{Y})^2\right)^{1/2}}.$$
Denoting by  $\psi_n = \psi_n(\mathbf{1},\bY) =\sum_{i=1}^nY_i/V_n$ the usual self-normalized partial sums, one can easily see that
 $$T_n=\psi_n  [(n-1)/(n-\psi_n^2)]^{1/2},$$
so that $\psi_n$ and  $T_n$ are equivalent (in terms of a 1:1 correspondence). See, e.g., Efron~\cite{efron}, Logan \emph{et al.} \cite{logan:etal} or Gin\'e \emph{et al.}~\cite{gine:goetze:mason} for a discussion.

More generally, we could phrase the above testing problem as $H_0:\,\beta = 0$ versus $H_1:\,\beta\neq 0$ in the linear model $Z_i= \beta X_i +Y_i$. (The setup $X_i=1$ for all $1\leq i\leq n$ is contained as a special case.) Then $\psi_n(\bX, \mathbf{Z})$, which
reduces under $H_0$ to $\psi_n(\bX, \bY)$, will serve as a natural test statistic.
As a matter of fact our research was originally motivated by this problem (see Hallin \emph{et al.} \cite{hallin:etal}). We were interested in obtaining asymptotic normality of this test under as
general as possible assumptions on the
errors $Y_i$.

Another related example where $\psi_n$ appears is the empirical correlation coefficient. If the sequences $\bX$ and $\bY$ are centered then
the empirical correlation is
$$
\rho_n(\bX,\bY)=\psi_n(\bX,\bY)/B_n=\frac{ \sum_{k=1}^n X_k Y_k}{B_n V_n},\quad  B_n=\sqrt{X_1^2+\cdots+X_n^2}.
$$
We will see how, under moment conditions on $\bX$, convergence rates for $\psi_n$ can be transfered to $\rho_n$ (see Lemma \ref{l:phi} below).

Besides their statistical applications, self-normalized sums have proven to be challenging mathematical objects with interesting properties. As a consequence they have attracted considerable attention in
probability theory.  For example Logan \emph{et al.}~\cite{logan:etal} studied the limiting distributions of $\psi_n(\mathbf{1},\bY)$ when $\bY$ is a centered  i.i.d.\ sequence with heavy tails, and conjectured that  $\psi_n(\mathbf{1},\bY)$ is asymptotically normal if and only if  $\bY$ is in the domain of attraction of the normal  law. Gin\'e \emph{et al.}~\cite{gine:goetze:mason} proved that this conjecture holds true, while  Chistyakov and G\"otze \cite{chistyakov:gotze} settled the question of the convergence of Student's statistic  by giving  necessary and sufficient conditions for these sums to allow limiting distributions which are not concentrated on $\{\pm1\}$.   More recently Benktus \emph{et al.} \cite{bentkus:etal2} studied the limiting distribution of the non-central  $t$-statistic under different assumptions on $\bY$; they show, inter alia,  how this limiting distribution depends critically on the existence of fourth moments for the $Y_i$. For a comprehensive study of these and related questions we refer the reader to the book Lai \emph{et al.} \cite{lai:etal}.

In a slightly different setup, Breiman \cite{breiman} provides necessary and sufficient conditions  for the weak convergence of randomly weighted self-normalized sums of the form $\sum_i X_iY_i/\sum_i Y_i$.  Mason and Zinn~\cite{mason:zinn} settle several questions left open by Breiman \cite{breiman}, and deduce the asymptotic distribution of $\psi_n(\mathbf{1},\bY)$ in the case of symmetry.

In this paper we will be interested in the rate of convergence to the normal distribution of
$\psi_n(\bX,\bY)$ as well as of $\rho_n(\bX,\bY)$. The case when $\bX=\mathbf{1}$ and $\{Y_i\}$ are independent with finite
variance is already well established.
Bentkus \emph{et al.}~\cite{bentkus:etal} give sharp rates for convergence
of Student's statistic, and thus equivalently for $\psi_n(\mathbf{1},\bY)$, in the non-i.i.d.\ case. Explicit constants
in these bounds were derived by Shao~\cite{shao:2005}.
See also Bentkus and G\"otze~\cite{bentkus:goetze} for further references.
There seem to be no similar investigations for $\psi_n(\bX,\bY)$. To the best of our knowledge, no similar results for the convergence rate
of the correlation coefficient $\rho_n(\bX,\bY)$ exist.


Our approach is as follows. We first state in Lemma~\ref{th:main} a general result, which is simple to prove and which provides a bound that holds without \emph{any} hypothesis on $\bY$, be it on its moments or dependence structure.  The main target is then to work out the thus obtained rates explicitly by imposing different assumptions on the sequence $\{Y_i\}$. This is done through a number of subsequent results. An interesting feature in our approach is that (with one exception) we do not work with truncation arguments, even when assuming an infinite variance for the $Y$'s.

\section{Results}

Recall that if $P$ and $Q$ are any probability measures on the real line, then
the \emph{Wasserstein distance} is given by
$$
d_W(P,Q)=\sup_{h\in\mathcal{H}}\left|\int hdP-\int hdQ\right|,
$$
where $\mathcal{H}$ is the class of Lipschitz 1 functions, i.e.\
$\mathcal{H}=\{h:\mathbb{R}\to\mathbb{R};\|h'\|\leq 1\}$ with $\|f\|=\sup_{x\in\mathbb{R}}|f(x)|$.
The \emph{Kolmogorov distance} $d_K(P,Q)$ is defined similarly, with $\mathcal{H}$ replaced by the class of
indicator functions $h_z(\cdot)=I\{\cdot\leq z\}$, $z\in\mathbb{R}$.
 If $V$ and $S$ are random
variables on the space $(\Omega,\mathcal{A},P)$ then
$d_W(V,S)$ will be written for $d_W(P\circ V^{-1},P\circ S^{-1})$,
 where $P\circ V^{-1}$ is the image measure of $V$ under $P$. Similar is the definition
for $d_K(V,S)$.

Throughout $Z$ stands for a standard normal random variable and
we are interested in
$$
d_W\big(\psi_n(\bX,\bY),Z\big)\quad\text{and}\quad
d_W\big(\rho_n(\bX,\bY),Z\big),
$$
and
$$
d_K\big(\psi_n(\bX,\bY),Z\big)\quad\text{and}\quad
d_K\big(\rho_n(\bX,\bY),Z\big),
$$
under the assumption that $\bX$ and $\bY$ are independent. 

The following simple Lemma gives the first step in our approach. 

\begin{lemma}\label{th:main}
Let $\psi_n(\bX,\bY)$ be defined as in \eqref{psi}, where $\bX$ and
$\bY$ are two mutually independent sequences. Assume that $\{X_k\}$ is  i.i.d.\  with
$EX_1=0$, $EX_1^2=1$, $\xi_3=E|X_1|^3<\infty$. Then
\begin{equation}\label{d:kol}
d_K\big(\psi_n(\bX,\bY),Z\big)\leq 0.56\,\xi_3\Delta,
\end{equation}
where
\begin{equation}\label{deltaY}
\Delta = \sum_{k=1}^nE|\delta_{k,n}|^3\quad \text{with}\quad \delta_{k,n}=Y_k/V_n.
\end{equation}
Furthermore
\begin{equation}\label{d:was}
d_W\big(\psi_n(\bX,\bY),Z\big)\leq \xi_3\Delta.
\end{equation}
\end{lemma}

\begin{proof}[Proof of Lemma~\ref{th:main}]
We show first \eqref{d:kol}.
Let $\bY_n=(Y_1,\ldots,Y_n)$, $F_n(\by_n)$ be the joint law of $\bY_n$ and set $v_n^2=\sum_{i=1}^ny_i^2$. Then using a version of the Berry-Esseen theorem for independent random variables, we obtain for any $z\in\mathbb{R}$
\begin{align*}
&|P(Z\leq z)-P(\psi_n(\bX,\bY)\leq z)|=\left|\int_{\mathbb{R}^n}P(Z\leq z)-P(\psi_n(\bX,\bY)\leq z|\bY_n=\by_n)dF_n(\by_n)\right|\\
&\quad\leq \int_{\mathbb{R}^n}\left|P(Z\leq z)-P(\psi_n(\bX,\by_n)\leq z)\right|dF_n(\by_n)\\
&\quad\leq C E|X_1|^3\int_{\mathbb{R}^n}\sum_{i=1}^n (y_i/v_n)^{3}dF_n(\by_n)=C\xi_3E\Delta.
\end{align*}
By a recent result of Shevtsova~\cite{shevtsova:2010}, $C\leq 0.56$.

The proof of \eqref{d:was} can be done in the exact same way, using
Corollary~4.2 in \cite{chen:goldstein:shao}.

\end{proof}

We remark that in Lemma~\ref{th:main} we do not put \emph{any} restrictions on the sequence $\{Y_k\}$.
This means that, in theory,  we can obtain non-trivial bounds even if this sequence is not independent or identically distributed.   Of course, the difficulty then resides in working out $\Delta$ explicitly, which we do under different assumptions in Theorems  \ref{l:fin3}, \ref{l:fin} and \ref{l:inf} below.
We will see that  Lemma~\ref{th:main}
provides optimal bounds in several special cases.

Let us consider first the following special case, which gives an application to self-normalized sums
$\psi_n(\mathbf{1},\mathbf{Y})$ when the $Y_i$ are not necessarily independent nor identically distributed.
We assume instead  that 
\begin{equation}\label{sign}
(Y_1, \ldots, Y_n) \stackrel{d}{=} (\pm Y_1, \ldots, \pm Y_n)
\end{equation}
for all choices of $+, -$. This form of symmetry, known as \emph{sign-symmetry}, is more general than spherical symmetry (see e.g. Serfling \cite{serfling:2006}) and is to be likened with  the concept of \emph{orthant symmetry} discussed by Efron \cite{efron}. Sign-symmetry is obviously satisfied  if the $Y_i$ are symmetric and independent random variables.  Under this condition  the following result (which should be also compared to Mason and Zinn \cite[Corollary 6]{mason:zinn}) holds.

\begin{corollary}\label{c:symm} Assume that \eqref{sign} holds and
 set $S_n=\sum_{k=1}^nY_k$. Then
 \begin{equation*}
d_W\big(S_n/V_n,Z\big)\leq \Delta \quad\text{and}
\quad d_K\big(S_n/V_n,Z\big)\leq 0.56\,\Delta,
\end{equation*}
with $\Delta= \sum_{k=1}^nE|\delta_{k,n}|^3$.
\end{corollary}

The proof follows simply by applying Lemma~\ref{th:main} to  $\psi_n(\bX,\bY)$ with
$\{X_k\}$ i.i.d.\ Rademacher variables, i.e. $X_k=\pm 1$ with probability $1/2$.
Then due to the symmetric distribution of
the $Y_k$,  we have that $S_n/V_n$
and $\psi_n(\bX,\bY)$ have the same distribution.

The next Lemma gives a simple criterion for switching from $\psi_n(\bX,\bY)$ to $\rho_n(\bX,\bY)$. While
we impose 4 moments for $X_1$, we keep the assumptions on $Y_1$ general.

\begin{lemma}\label{l:phi} Let the assumptions of Lemma~\ref{th:main} hold and
assume in addition that $m_4:=EX_1^4<\infty$. Then, if $n/\log n\geq 8m_4$,
$$
d_W(\psi_n(\bX,\bY),\sqrt{n}\rho_n(\bX,\bY))\leq \sqrt{\frac{2m_4}{n}}\,.
$$
\end{lemma}

Let us consider once more the testing problem $H_0:\,\beta = 0$ versus $H_1:\,\beta\neq 0$ in the linear model $Z_i= \beta X_i +Y_i$.
The previous lemma in connection with Lemma~\ref{th:main} shows, if the regressors $X_i$ are centered (a condition
which is convenient but could be modified) and have 4 moments then we get under $H_0$ for very general errors $Y_i$  the convergence of 
the correlation test statistic $\sqrt{n}\rho_n(\bX,\mathbf{Z})$ to the normal, with an approximation error of order $O(n^{-1/2}+\Delta)$.

When $\{Y_k\}$ is a stationary sequence, then $\Delta=nE|\delta_{1,n}|^3$ and
obtaining a rate of convergence to the normal distribution is entirely reduced
to calculating the third absolute moment of $\delta_{1,n}$.
We now concentrate on obtaining $\Delta$ under different moment and tail assumptions on the sequence $\{Y_k\}$ under the i.i.d.\ setup.
We first work out $\Delta$ under the sole assumption $E|Y_1|^p<\infty$, $p\in (2,3]$. In this case we obtain the ``usual''
convergence rates.
\begin{theorem}\label{l:fin3} Let $\{Y_i\}$ be an i.i.d.\ sequence, let $p\in (2,3]$ and assume $EY_1^2=1$
and $E|Y_1|^p<\infty$. Then
\begin{equation}\label{p23}
\Delta=nE\left|\delta_{1,n}\right|^3\leq nE\left|\delta_{1,n}\right|^p\sim
E|Y_1|^p\,n^{1-p/2}.
\end{equation}
\end{theorem}
(Note that the first inequality in \eqref{p23} follows from $|\delta_{1,n}|\leq 1$.)

\begin{remark}
A look at the proof of Theorem~\ref{l:fin3} suggests that similar results may be obtained under different dependence conditions, too. In fact, besides some purely analytic estimates, which
hold for any sequence $\{Y_k\}$, we only make use of moment inequalities which exist in
different generality for many weak dependence and mixing concepts, respectively.
\end{remark}

Next we consider the case when we have knowledge on the tail probabilities of the $Y_k$. Let $\Gamma(p)$ denote Euler's gamma function. 
\begin{theorem}\label{l:fin} Let $\{Y_i\}$ be an i.i.d.\ sequence, let $1\leq \alpha<2$ and $P(Y_k^2>x)\sim\ell(x)x^{-\alpha}$,
where $\ell(x)$ is slowly varying at $\infty$. If $\sigma_Y^2:=EY_1^2<\infty$,
then we have for any $\gamma>\alpha$
$$
nE\left|\delta_{1,n}\right|^{2\gamma}\sim
\frac{\Gamma(\gamma-\alpha)\Gamma(1+\alpha)}{\sigma_Y^2\Gamma(\gamma)}\,n^{1-\alpha}\ell(n).
$$
\end{theorem}

\begin{example}
Consider the case $P(|Y_1|^p>x)\sim \frac{1}{x\log^2 x}$, $p\in (2,3)$.
Then $E|Y_1|^\beta=\infty$ for any $\beta>p$, while $E|Y_1|^p<\infty$. Applying the
above result with $\gamma=3/2$ we obtain
$$
\Delta=nE\left|\delta_{1,n}\right|^3\sim
\frac{\Gamma((3-p)/2)\Gamma(1+p/2)}{\sigma_Y^2\Gamma(3/2)}\,\frac{n^{1-p/2}}{\log^2n}.
$$
Hence the additional knowledge of the tail behavior  yields a slightly better rate than the one obtained
in \eqref{p23}.
\end{example}

We now turn to the case when we have infinite second moments. 

\begin{theorem}\label{l:inf} Let $\{Y_i\}$ be an i.i.d.\ sequence, let $P(Y_1^2>x)\sim\ell(x)x^{-1}$,
with $\ell(x)$ slowly varying at $\infty$. If $E(Y_1^2)=\infty$,
then, for any $\gamma>1$, we have
$$
nE\left|\delta_{1,n}\right|^{2\gamma}\sim\frac{1}{\gamma-1}\frac{\ell(a_n)}{L(a_n)},
$$
where $L(x)=\int^x\big(\ell(t)/t\big)dt$ and $\{a_n\}$ is a sequence satisfying
$a_n\sim n L(a_n)$.
\end{theorem}

\begin{example} Assume  that $P(Y_1^2>x)\sim x^{-1}(\log x)^{-2}$. Then $EY_1^2<\infty$ and
by Theorem~\ref{l:fin} we get that for all $n\geq 1$
$$
d_K\big(\psi_n(\bX,\bY),Z\big)\leq A(\log n)^{-2},
$$
for some large enough constant $A$.
If $P(Y_k^2>x)\sim x^{-1}(\log x)^{-1}$ then $EY_1^2=\infty$. But since $\tilde\ell(x)=
\log\log x$
and $a_n\sim n\log\log n$ we still get by Theorem~\ref{l:inf} applied with $\gamma=3/2$
$$
d_K\big(\psi_n(\bX,\bY),Z\big)\leq A
\big((\log n)\log\log n\big)^{-1},
$$
and thus an explicit convergence rate to the normal law.
\end{example}


We conclude with a result which shows that, even in the case when $Y_1$ is in the domain of attraction of an $\alpha$-stable law
with $\alpha$ strictly less but close to $2$ we can get non-trivial bounds.
\begin{theorem}
Let $\{Y_i\}$ be an i.i.d.\ sequence, assume that $P(|Y_1|>x)\sim \ell(x)x^{-\alpha}$ with $\alpha\in (0,2)$. Then for any $\gamma>1$
\begin{equation}\label{p24}
nE\left|\delta_{1,n}\right|^{2\gamma}\sim \dfrac{\Gamma(\gamma-\alpha/2)}{\Gamma(\gamma)\Gamma(1-\alpha/2)}.
\end{equation}
\end{theorem}

\begin{example}
We apply this result with $\gamma=3/2$. Let $\alpha=2-\varepsilon$ for small $\varepsilon>0$.
Observing that under the above assumptions $\Gamma(3/2-\alpha/2)/\Gamma(3/2)< 2$, and $1/\Gamma(\varepsilon)\sim\varepsilon$ for
$\varepsilon\to 0$ we get by Lemma~\ref{th:main} that  for large enough $n$
$$
d_K(\psi_n(\bX,\bY),Z)\leq 0.56\,\xi_3\,\varepsilon.
$$
When the distribution of the $Y_k$ is symmetric, we can conclude that for sufficiently large sample size $n$ we have
$
d_K(S_n/V_n,Z)\leq 0.56\,\varepsilon.
$
\end{example}

\section{Proofs}

In the sequel  we need the following version of Hoeffding's inequality (see e.g.\ Shao~\cite[p. 145]{shao:2005}).
\begin{lemma}\label{l:hoeffding}
Let $\{Z_i,\,1\leq i\leq n\}$ be independent non-negative random variables with
$\mu=\sum_{i=1}^nEZ_i$ and $\sigma^2=\sum_{i=1}^nEZ_i^2<\infty$. Then for
$0<x<\mu$
$$
P\left(\sum_{i=1}^nZ_i\leq x\right)\leq\exp\left(-\frac{(\mu-x)^2}{2\sigma^2}\right).
$$
\end{lemma}

\begin{proof}[Proof of Lemma~\ref{l:phi}]
Let
\begin{equation*}
\phi := \sqrt{n}\rho_n(\bX,\bY).
\end{equation*}
 Note that $\phi=\sqrt{\frac{n}{B_n^2}}\psi$. Then, for all $h\in\mathcal{H}$,  by the mean value theorem  we get (recall that $\|h'\|\leq 1$)
$$
\left|
Eh(\phi)-Eh(\psi)
\right|\leq E\left[\left|\sqrt{\frac{n}{B_n^2}}-1\right||\psi|\right].
$$
Let $\varepsilon\in (0,1)$. Then the last term is bounded by $\frac{1}{2}A_n^{(1)}+A_n^{(2)}$, where
$$
A_n^{(1)}:=E\left[\left|\frac{n}{B_n^2}-1\right||\psi|I\{B_n^2>(1-\varepsilon)n\}\right]
$$
and
$$
A_n^{(2)}:=
 E\left[\left|\sqrt{\frac{n}{B_n^2}}-1\right||\psi|I\{B_n^2\leq(1-\varepsilon)n\}\right].
$$
Define $\kappa_4=E(X_1^2-1)^2$. Then
\begin{align*}
A_n^{(1)}&\leq\frac{1}{1-\varepsilon}E\left[\left|\frac{1}{n}\sum_{i=1}^n(X_i^2-1)\right||\psi|\right]\\
&\leq \frac{1}{1-\varepsilon}\left(E\left[\bigg|\frac{1}{n}\sum_{i=1}^n(X_i^2-1)\bigg|^2\right]\right)^{1/2}\times \left(E|\psi|^2\right)^{1/2}\\
&= \frac{1}{1-\varepsilon}\sqrt{\frac{\kappa_4}{n}}.
\end{align*}
For estimating $A_n^{(2)}$ we use
$$
|\psi|=\left|\sum_{i=1}^nX_i\frac{Y_i}{V_n}\right|\leq \left(\sum_{i=1}^nX_i^2\right)^{1/2}\left(\sum_{i=1}^n\frac{Y_i^2}{V_n^2}\right)^{1/2}=B_n.
$$
Thus
\begin{align*}
A_n^{(2)}&\leq E|\sqrt{n}-B_n|I\{B_n^2\leq (1-\varepsilon)n\}\leq \sqrt{n}P(B_n^2\leq (1-\varepsilon)n).
\end{align*}
By Lemma~\ref{l:hoeffding} we get that
$P(B_n^2\leq (1-\varepsilon)n)\leq \exp(-\varepsilon^2n/(2m_4))$.
Collecting our estimates we have
$$
\left|
Eh(\phi)-Eh(\psi)
\right|\leq\frac{1}{2}\frac{1}{1-\varepsilon}\sqrt{\frac{\kappa_4}{n}}+
\sqrt{n}\exp(-\varepsilon^2n/(2m_4)).
$$
For large enough $n$ we have $\varepsilon^2:=(2m_4\log n)/n\leq 1/4$, and since
$m_4=\kappa_4+1$ we conclude
$$
\left|
Eh(\phi)-Eh(\psi)
\right|\leq \sqrt{\frac{\kappa_4}{n}}+\frac{1}{\sqrt{n}}\leq \sqrt{\frac{2m_4}{n}}.
$$
\end{proof}

\begin{proof}[Proof of Theorem~\ref{l:fin3}]
Let $\tilde{Y}_{i,n}=Y_i\wedge n^{1/p}$. Then
\begin{equation}\label{tildeY}
E\tilde{Y}_{i,n}^2=1-\varepsilon_n\quad\text{with $\varepsilon_n\to 0$.}
\end{equation}
Further
\begin{align}
E\tilde{Y}_{i,n}^4&=\int_0^{n^{1/p}}x^4dP(Y_1\leq x)\leq n^{2/p}\int_0^{\infty}x^2dP(Y_1\leq x)= n^{2/p}.\label{tildeY2}
\end{align}
Now fix an arbitrarily small $\varepsilon>0$ and let $n$ be  large enough in
order to have $\varepsilon_n\leq \varepsilon/2$.
Using Lemma~\ref{l:hoeffding} with \eqref{tildeY} and \eqref{tildeY2}
it follows that
\begin{equation}\label{tildeY3}
P\left(\tilde{Y}_{2,n}^2+\cdots+ \tilde{Y}_{n,n}^2\leq (1-\varepsilon)(n-1)\right)
\leq\exp\left(-\frac{\varepsilon^2}{8}(n-1)^{1-2/p}\right).
\end{equation}
Next we observe that
\begin{align*}
&\left(\frac{Y_1^2}{Y_1^2+\cdots+Y_n^2}\right)^{p/2}
\leq\mathrm{min}\left\{1,\frac{|Y_1|^p}{(\tilde{Y}_{2,n}^2+\cdots+ \tilde{Y}_{n,n}^2)^{p/2}}\right\}\\
&\qquad\leq
I\{\tilde{Y}_{2,n}^2+\cdots+ \tilde{Y}_{n,n}^2\leq (1-\varepsilon)n\}\\
&\qquad\quad+|Y_1|^p\left(\frac{1}{n(1-\varepsilon)}\right)^{p/2}
I\{\tilde{Y}_{2,n}^2+\cdots+ \tilde{Y}_{n,n}^2> (1-\varepsilon)n\}
\end{align*}
This and \eqref{tildeY3} give
\begin{align}
nE|\delta_{1,n}|^p&=nE\left(\frac{Y_1^2}{\sum_{k=1}^nY_k^2}\right)^{p/2}\nonumber\\
&\leq nP(\tilde{Y}_{2,n}^2+\cdots+ \tilde{Y}_{n,n}^2\leq (1-\varepsilon)n)
+E|Y_1|^p\left(\frac{1}{1-\varepsilon}\right)^{p/2}n^{1-p/2}\nonumber\\
&\sim E|Y_1|^p\left(\frac{1}{1-\varepsilon}\right)^{p/2}n^{1-p/2}\quad (n\to\infty).
\label{lastf}
\end{align}
On the other hand we have
\begin{align}
&nE\left(\frac{Y_1^2}{\sum_{k=1}^nY_k^2}\right)^{p/2}\nonumber\\
&\quad\geq E|Y_1|^pI\{|Y_1|^p<n\}E\left(\frac{1}{n^{-1}
\sum_{k=2}^nY_k^2+n^{(2-p)/p}}\right)^{p/2}n^{1-p/2}\nonumber\\
&\quad\geq E|Y_1|^pI\{|Y_1|^p<n\}\left(\frac{1}{1+n^{(2-p)/p}}\right)^{p/2}
P\left(\frac{1}{n}\sum_{k=2}^nY_k^2\leq 1\right)n^{1-p/2}\nonumber\\
&\quad\sim E|Y_1|^pn^{1-p/2}\quad
(n\to \infty),\label{last}
\end{align}
where in the last step we used the law of large numbers to obtain $P(\frac{1}{n}\sum_{k=2}^nY_k^2\leq 1)\to 1$. Now \eqref{lastf} holds
for arbitrarily small $\varepsilon>0$. Together with \eqref{last}
the proof follows.
\end{proof}

\begin{proof}[Proof of Theorem~\ref{l:inf}]

We borrow an idea of Albrecher and Teugels~\cite{albrecher:teugels}.
The crucial trick is to write
$$
\frac{1}{x^\gamma} = \frac{1}{\Gamma(\gamma)}\int_0^\infty s^{\gamma-1}e^{-sx}ds,\quad \gamma>0.
$$
Then, since
$$
E\big|\delta_{1,n}^2\big|^\gamma  =E\left|\frac{Y_1^2}{\sum_{k=1}^nY_k^2}\right|^{\gamma},
$$
we obtain
$$
E\big|\delta_{1,n}^2\big|^\gamma   = \frac{1}{\Gamma(\gamma)}\int_0^\infty s^{\gamma-1}\left(\varphi_1(s)\right)^{n-1}\varphi_2(s)ds,
$$
with $\varphi_1(s) := E\big(e^{-sY_1^2}\big)$ and $\varphi_2(s)
:= E\left([Y_1^2]^\gamma  e^{-s Y_1^2}\right)$.  Choosing $a_n\to\infty$ such that $na_n^{-1}L(a_n)\to1$ one easily shows (see \cite{albrecher:teugels}) that
for any $s>0$
\begin{equation}\label{eqdeltan1}
\lim_{n\to\infty}\varphi_1^{n-1}\left(\frac{s}{a_n}\right)=e^{-s}.
\end{equation}
In order to determine $\varphi_2(s)$ we introduce
$$
\Gamma_{\gamma,s}(x)=\int_0^xt^\gamma e^{-st}dt,\quad \gamma>1,\quad s,x>0.
$$
Note that $\lim_{x\to\infty}\Gamma_{p,s}(x)=s^{-(\gamma+1)}\Gamma(\gamma+1)$.
Further we let $F$ be the distribution function of $Y_1^2$.
Using integration by parts, we get
$$
\int_0^\infty F(x)d\Gamma_{\gamma,s}(x)=s^{-(\gamma+1)}\Gamma(\gamma+1)-\int_0^\infty\Gamma_{\gamma,s}(x)dF(x).
$$
A simple consequence is that
$$
\int_0^\infty \Gamma_{\gamma,s}(x)dF(x)=\int_0^\infty(1-F(x))d\Gamma_{\gamma,s}(x)=\int_0^\infty x^\gamma e^{-sx}(1-F(x))dx.
$$
Since $\gamma\Gamma_{\gamma-1,s}(x)-s\Gamma_{\gamma,s}(x)=x^\gamma e^{-sx}$ we conclude that
\begin{align*}
\varphi_2(s)&=\int_0^\infty x^\gamma e^{-sx}dF(x)\\
&=\gamma\int_0^\infty \Gamma_{\gamma-1,s}(x)dF(x)-s\int_0^\infty\Gamma_{\gamma,s}(x) dF(x)\\
& = \gamma\int_0^\infty x^{\gamma-1}(1-F(x))e^{-sx}dx-s\int_0^\infty x^\gamma(1-F(x))e^{-sx}dx.
\end{align*}
By our assumption $1-F(x)\sim x^{-1}\ell(x)$ and thus by
Karamata's Tauber theorem (see e.g.\ Bingham \emph{et al.}~\cite[Theorem 1.7.6]{bgt:1987}) we have for any $\rho>-1$
$$
\int_0^\infty x^{\rho}(1-F(x))e^{-sx}dx\sim \Gamma(\rho)s^{-\rho}\ell(1/s)\quad
\text{as}\quad s\to 0.
$$
Combining with our just derived formula for $\varphi_2(s)$
 we have for $\gamma >1$
\begin{equation}\label{eqdeltan2} \varphi_2(s) \sim \frac{\ell(1/s)}{s^{\gamma -1}}\Gamma(\gamma -1)\quad \text{as $s\to 0$.}
\end{equation}

Finally consider the quantity
$$nE\big|\delta_{1,n}^2\big|^\gamma   = \frac{n}{\Gamma(\gamma)}\int_0^{\infty}t^{\gamma -1}\varphi_1^{n-1}(t)\varphi_2(t)dt.$$
It is easy to show that  $n\int_\epsilon^{\infty}t^{\gamma -1}\varphi_1^{n-1}(t)\varphi_2(t)dt\to 0$ for all $\epsilon>0$. Hence we can restrict the integration to the compact interval $[0,\epsilon]$, on which Lemma 5.1 of Fuchs \emph{et al.} \cite{fuchs:etal} can be used to uniformly bound the integrand above by an integrable function. Using the already defined $a_n$,  we therefore get from (\ref{eqdeltan1}), (\ref{eqdeltan2}) and dominated convergence
\begin{align*}nE\big|\delta_{1,n}^2\big|^\gamma & \sim \frac{n}{\Gamma(\gamma)}\int_0^{\epsilon}t^{\gamma-1}\varphi_1^{n-1}(t)\varphi_2(t)dt\\
& \sim \frac{n}{\Gamma(\gamma)}\left(\frac{1}{a_n}\right)^\gamma\int_0^{\infty}t^{\gamma-1}\varphi_1^{n-1}(t/a_n)\varphi_2(t/a_n)dt\\
& \sim  \frac{\Gamma(\gamma-1)}{\Gamma(\gamma)} \frac{n \ell(a_n)}{a_n},\quad\text{as $n\to \infty$.}
\end{align*}
The relation $\frac{n \ell(a_n)}{a_n}\sim \frac{\ell(a_n)}{L(a_n)}$ concludes the proof.
\end{proof}

The proof of Theorem~\ref{l:fin} is similar to the proof of Theorem~\ref{l:inf}
and will therefore be omitted.

\subsection*{Acknowledgments.} The authors thank Marc Hallin and Thomas Verdebout for providing the incentive for this paper.

\


\begin{thebibliography}{}
\bibitem{albrecher:teugels}
Albrecher, H., and Teugels, J. (2006).
\newblock Asymptotic analysis of a measure of variation.
\newblock {\em Theor.\ Probab.\ Math.\ Stat.}, {\bf 74}{\bf ,} 1--9.

\bibitem{bentkus:etal}
Bentkus, V., Bloznelis, M., and G\"otze, F. (1996).
\newblock A Berry-Ess\'een bound for Student's statistic in the non-i.i.d.\ case.
\newblock {\em J.\ Theor.\ Probab.}, {\bf 9}{\bf ,} 765--796.

\bibitem{bentkus:etal2}
Bentkus, V., Bing-Yi, J., M., Shao, Q.M. and Wang, Z.  (2007).
\newblock Limiting distributions of the non-central $t$-statistic and their applications to the power of $t$-tests under non-normality.
\newblock {\em Bernoulli}, {\bf 13}{\bf ,} 346--364.


\bibitem{bentkus:goetze}
Bentkus, V., and G\"otze, F. (1996).
\newblock The Berry-Ess\'een bound for Student's statistic.
\newblock {\em Ann.\ Probab.}, {\bf 24}{\bf ,} 491--503.

\bibitem{bgt:1987}
Bingham, N.H., Goldie, C.M., and Teugels, J.L. (1987).
\newblock {\em Regular Variation}.
\newblock { Cambridge University Press.}


\bibitem{breiman}
Breiman, L. (1965).
\newblock On some limit theorems similar to the arc-sin law.
\newblock {\em Teor. Veroyatnost. i Primenen.}, {\bf 10}{\bf ,} 351--359.


\bibitem{chen:shao}
Chen, L.H.Y, and Shao, Q-M. (2004).
\newblock Stein's method for normal approximation.
\newblock In {\em An Introduction to Stein's Method}
 (A.D. Barbour and L.H.Y. Chen eds). Lecture Notes Series, Institute for Mathematical Sciences, NUS, Vol. 4, p. 1-59.

\bibitem{chen:goldstein:shao}
 Chen, L.H.Y, Goldstein, L.\ and Shao, Q-M. (2011).
\newblock Normal Approximation by Stein's Method. {\em Springer Series in Probability and its Application.}

\bibitem{chistyakov:gotze}
Chistyakov, G. P. and G\"otze, F. (2004).
\newblock Limit distributions of studentized means.
\newblock {\em Ann.\ Probab.}, {\bf 32}{\bf ,}  No. 1 A, 28--77.

\bibitem{efron}
Efron,  B. (1969).
\newblock Student's $t$-test under symmetry conditions.
\newblock{\em JASA},  {\bf  64}{\bf ,} 1278--1302.


\bibitem{fuchs:etal}
Fuchs, A., Joffe, A. and Teugels, J. (2001).
\newblock Expectation of the Ratio of the Sum of Squares to the Square of the Sum: Exact and Asymptotic Results.
\newblock{\em Teor. Veroyatnost. i Primenen},  {\bf  46}{\bf ,} 297--310.




\bibitem{gine:goetze:mason}
Gin\'e, E., G\"otze, F., and Mason, D.M. (1997).
\newblock When is the student $t$-statistic asymptotically normal?
\newblock {\em Ann.\ Probab.}, {\bf 25,} 1514--1531.

\bibitem{hallin:etal}
Hallin, M., Swan, Y., Verdebout, T. and  Veredas, D. (2010).
\newblock Rank-based Inference in Linear Models with Stable Errors.
\newblock{\em  JNPS}, forthcoming.

\bibitem{lai:etal}
Lai, T.L., de la Pena, V. and Shao, Q. M. (2009).
\newblock {\em Self-normalized Processes: Theory and Statistical Applications.}
\newblock Springer Series in Probability and its Applications,  Springer-Verlag, New York.


\bibitem{logan:etal}
Logan, B.F., Mallows, C.L., Rice, S.O., and Shepp, L.A. (1973).
\newblock Limit distributions of self-normalized sums.
\newblock {\em Ann.\ Probab.}, {\bf 5}{\bf ,} 788--809.

\bibitem{mason:zinn}
Mason, D., and Zinn, J. (2005).
\newblock When does a self-normalized weighted sum converge in distribution?
\newblock {\em Elec.\ Comm.\ in Probab.}, {\bf 10}{\bf ,} 70--81.

\bibitem{serfling:2006}
Serfling, R. (2006). Multivariate symmetry and asymmetry. \newblock In: {\em  Encyclopedia of
Statistical Sciences, 2nd Ed. (Kotz, Balakrishnan, Read and Vidakovic, Eds.).} Wiley,  5338--5345.

\bibitem{shao:2005}
Shao, Q-M. (2005).
\newblock
An explicit Berry-Esseen bound for the Student t-statistic via Stein's method.
\newblock In:
{\em Stein's Method and Applications (A.D. Barbour and L.H.Y. Chen eds).} Lecture Notes Series, Institute for Mathematical Sciences, NUS, Vol.\ 5, 143--155.

\bibitem{shevtsova:2010}
Shevtsova, I.G. (2010).
\newblock
An improvement of convergence rate estimates in the Lyapunov theorem.
\newblock {\em Doklady Mathematics}, {\bf 82}{\bf ,} 862--864.
\end{thebibliography}
\end{document}